\documentclass[11pt]{amsart}
\usepackage[latin1]{inputenc}
\usepackage{graphicx}
\usepackage[T1]{fontenc}
\usepackage{amssymb,pb-diagram}
\usepackage{amsmath}
\usepackage{amsthm,amssymb,amsfonts}
\usepackage{amssymb,amscd}
\usepackage{epic,eepic,epsfig}
\usepackage{latexsym}
\usepackage{stmaryrd}
\usepackage{xy}
\usepackage{stmaryrd}
\usepackage{amssymb,amsthm,enumerate,color,mathrsfs}
\xyoption{all}

\usepackage{amssymb,amsthm,enumerate,color,mathrsfs}

\newtheorem{Theorem}{Theorem}[section]

\newtheorem{Lemma}[Theorem]{Lemma}
\newtheorem{Corollary}[Theorem]{Corollary}

\theoremstyle{definition}
\newtheorem{Definition}[Theorem]{Definition}
\newtheorem{Remark}[Theorem]{Remark}

\newtheorem{Examples}[Theorem]{Examples}

\theoremstyle{definition}

\theoremstyle{remark}

\newcommand{\bib}{\bibitem}

\begin{document}

\renewcommand{\proofname}{Proof}

\title{Almost bi-Lipschitz embeddings and proper subsets of a Banach space - An extension of a Theorem by M.I. Ostrovskii}

\author{F. Netillard}
\address{Universit\'{e} de Franche-Comt\'{e}, Laboratoire de Math\'{e}matiques UMR 6623,
16 route de Gray, 25030 Besan\c{c}on Cedex, FRANCE.}
\email{francois.netillard2@univ-fcomte.fr}




\maketitle

\begin{abstract}
Let $X$ and $Y$ be two infinite-dimensional Banach spaces. If $X$ is crudely finitely representable in every finite-codimensional subspace of $Y$, then any proper subset of $X$ almost bi-Lipschitz embeds into $Y$, in a sense quite close to that of F. Baudier and G. Lancien (see \cite{B} and \cite{BL}). This is an extension of a result proved by M.I. Ostrovskii for locally finite subsets \cite{O}.
\end{abstract}

\mbox{
}
\mbox{
\hspace{.4cm}\small{\textbf{MSC:} 46B20, 46B85.}
}

\markboth{}{}

\section{Introduction}

In order to follow this work, here are some essential definitions to know:

\begin{Definition}
Let $X$ and $Y$ be two infinite-dimensional Banach spaces, and $C>1$.\\
We say that $X$ is $C$-\textit{crudely finitely representable} in $Y$ (in short $C$-c.f.r) if for any finite-dimensional subspace $E$ of $X$, there is a linear isomorphism from $E$ to the space $T(E)\subset Y$ verifying $\|T\|\cdot\|T^{-1}\|<C$.\\
When the value of $C$ does not matter, we will simply say that $X$ is \textit{crudely finitely representable} in $Y$ (in short c.f.r).
\end{Definition}

\begin{Definition}
We say that $X$ is \textit{finitely representable} in $Y$ if, for all $\varepsilon > 0$, $X$ is $(1+\varepsilon)$-c.f.r in $Y$ for all $\varepsilon>0$.
\end{Definition}

Note that the concept of finite representability is due to R.C. James (\cite{J3} and \cite{J4}).

\begin{Definition}
Let $X$ and $Y$ be two infinite-dimensional Banach spaces.
We say that $X$ is \textit{c.f.r in the finite-codimensional subspaces of} $Y$ if there exists a constant $C>1$ such that, for all $Z \in \textnormal{cof}(Y)$, $X$ is $C$-c.f.r in $Z$.
\end{Definition}

\begin{Remark}
Note that in the above definition, the constant $C$ is {\sl uniform} on all finite-codimensional subspaces. By Proposition 4.1 of \cite{F}, c.f.r in $Y$ does not imply c.f.r in the finite-codimensional subspaces of $Y$.
\end{Remark}

An important question in nonlinear Banach space theory is: if a Banach space $X$ is crudely finitely representable in a Banach space $Y$, what are the similarities between subsets of $X$ and subsets of $Y$~?\\
The nonlinear maps that we will use to study these similarities here are embeddings. It should be noted that embeddings of discrete metric spaces into Banach spaces recently became an important tool in computer science and topology. A familiar class in nonlinear Banach space theory is constituted by the bi-Lipschitz embeddings which are metric isomorphisms between the domain space and the range. \\
It is known that, in such cases a Banach space $X$ which is crudely finitely representable in another Banach space, $Y$ does not have to admit a bi-Lipschitz embedding into $Y$ even if $X$ is separable. For example, $L_p[0,1]$ is finitely representable in $\ell_p$ (see, for instance, Theorem 6.2 and its proof from \cite{Fa}) but does not admit a bi-Lipschitz embedding into $\ell_p$, where $1\leq p<\infty$, $p \neq 2$ (this last statement is a consequence of Theorem 1.3 from \cite{H}). On the other hand, M.I. Ostrovskii \cite{O} proved that each locally finite subset of $X$ (that is subset metric spaces of $X$ in which each ball of finite radius has finite cardinality) whose finite subsets admit uniformly bi-Lipschitz embeddings admits a bi-Lipschitz embedding into $Y$.\\
Our interest is in "thicker" subsets, and more precisely in proper subsets of $X$, i.e metric subspaces of $X$ whoses closed balls are compact. For them, we cannot expect to get bi-Lipschitz embeddings due to a technical net argument which requires separation parameters, inspired by the work of F. Baudier and G. Lancien \cite{BL}.
Along these lines then the condition of almost bi-Lipschitz embeddability (introduced in \cite{BL}) happens to be useful. The definition given here will be slightly different from that of \cite{BL} but, in both cases, the notion of almost bi-Lipschitz embeddability expresses the fact that one can construct an embedding that is as close as one wishes to a bi-Lipschitz embedding.\\
Two theorems (from \cite{BL} and \cite{O}) lie in the background of our main result.

${}$

First of all, we recall the theorem of M.I. Ostrovskii, already mentioned above \cite{O} (see also \cite{O2}).

\begin{Theorem}
Let $A$ be a locally finite metric space whose finite subsets admit uniformly bi-Lipschitz embeddings into a Banach space $X$. Then $A$ admits a bi-Lipschitz embedding into $X$.
\end{Theorem}

\noindent This theorem is quite close to the following result \cite{BL}, where the notion of crude finite representability shows up.

\begin{Theorem}\label{T111}
Let $X$ and $Y$ be two Banach spaces such that $X$ is crudely finitely representable in $Y$.\\
Let $A\subset X$ such that $A$ is locally finite.\\
Then $A$ admits a bi-Lipschitz embedding into $Y$.
\end{Theorem}



Theorem \ref{T111} yields to the conjecture that any proper subset of a Banach space $X$ almost bi-Lipschitz embeds into another Banach space $Y$ as soon as $X$ is crudely finitely representable in $Y$.
This motivated our work (which is an improvement of Theorem 4.2.3 from \cite{Ne}).
In connection with it, F. Baudier et G. Lancien proved in \cite{BL}, among other results, that, for $p \in [1,+\infty]$ and $Y$ a Banach space which contains uniformly the $\ell_p^n$ spaces, every proper subset $M$ of $L_p$ almost bi-Lipschitzly embeds into $Y$.
Baudier-Lancien's proof of this theorem provides a pattern for the  proof of Theorem \ref{P872} below, which is the main result of this paper.

For this Theorem, we give the following definition of an almost bi-Lipschitz embedding:

\begin{Definition}
Let $(M,d)$ and $(N,\delta)$ be two metric spaces.\\
$M$ is almost bi-Lipschitz embeddable into $N$ if for any $A \in (0,\infty)$, there exist a scaling factor $r \in (0,\infty)$ and a constant $D \in [1,\infty)$ such that for any continuous function $\varphi: [0,\infty) \rightarrow [0,1)$ with $\varphi(0)=0$ and $\varphi(t) >0~ \forall t>0$, there exists a map $f_{\varphi}: M \rightarrow N$ such that:
$$\forall~x,y \in M,~
\varphi(d(x,y))rd(x,y) \leq \delta(f_{\varphi}(x),f_{\varphi}(y)) \leq Drd(x,y)+A.$$
\end{Definition}

\begin{Theorem}\label{P872} Let $X$ and $Y$ be two Banach spaces.\\
Assume that $X$ is crudely finitely representable in the finite-codimensional subspaces of $Y$.
Let $M$ be a proper subset of $X$. Then $M$ almost bi-Lipschitz embeds into $Y$.
\end{Theorem}

\section{Banach spaces containing the $\ell_p^n$'s uniformly}






The definitions below are classical items from the local theory of Banach spaces.

\begin{Definition}
Let $\lambda \in [1,\infty)$ and $p \in [1,\infty]$.\\
We say that a Banach space contains $\lambda$-uniformly the $\ell_p^n$'s if there exists a sequence of subspaces $(E_n)_{n=1}^{\infty}$ of $X $ such that $E_n$ is linearly isomorphic to $\ell_p^n$ with
\begin{center}$d_{BM}(E_n,\ell_p^n)\leq \lambda$,\end{center}
where $d_{BM}$ is the Banach-Mazur distance between Banach spaces.
\end{Definition}

\begin{Definition}
Let $p \in [1,\infty]$.\\ We say that $X$ contains the $\ell_p^n$ 's uniformly if it contains $\lambda$-uniformly the $\ell_p^n$'s for some value $ \lambda \in [1,\infty)$.
\end{Definition}

\section{Main results}

The following lemma is shown in \cite{BL}.

\begin{Lemma}\label{L123} For any continuous function $\varphi:[0,+\infty)\rightarrow [0,1)$ such that $\varphi(0)=0$ and $\varphi(t)>0$ for all $t>0$, there exists a continuous non-decreasing surjective function $\mu:(0,+\infty)\rightarrow (-\infty,0)$, so that $\varphi(t)\leq 2^{\mu(t)}$ for all $t>0$.
\end{Lemma}

We will need the next lemma, inspired by Lemma 2.6 from \cite{BL}, to establish Theorem \ref{P872}. For this lemma and what follows, $d_{BM}$ denotes the Banach-Mazur distance, FDD is an abbreviation for finite-dimensional decomposition, and Span is a notation for linear span.

\begin{Lemma}\label{L358} Let $X$ and $Y$ be two Banach spaces. We assume that $X$ is f.c.r in the finite-codimensional subspaces of $Y$.\\
Let $\gamma > 0$, $(F_j)_{j=1}^{\infty}$ be a family of finite dimensional subspaces of $X$.\\
Then there exists a family $(H_j)_{j=1}^{\infty}$ of finite dimensional subspaces of $Y$ such that, for all $j$, $d_{BM}(H_j,F_j ) \leq C$ and $(H_j)_{j=1}^{\infty}$ is an FDD of $Z=\overline{\textnormal{Span}}(\bigcup_{j=1}^{\infty }H_j) \subseteq Y$ with $\|P_j\| \leq 1+\gamma$ where $P_j$ is the projection of $Z$ onto $H_1\oplus \cdots \oplus H_j$ whose kernel is $\overline{\textnormal{Span}}(\bigcup_{i= j+1}^{\infty}H_i)$.
\end{Lemma}

\begin{proof}
Pick a sequence $(\gamma_j)_{j=1}^{\infty}$ with $\gamma_j>0$ and $\prod\limits_{j=1}^{\infty}(1+\gamma_j)\leq 1+\gamma$.\\
We proceed by induction on $j\in\mathbb{N}$.\\
Since $X$ is $C$-crudely finitely representable in $Y$, there exists a finite-dimensional subspace $H_1$ of $Y$ so that $d_{BM}(H_1,F_1)\leq C$.\\
Using the standard Mazur technique, we can find a finite-codimensional subspace $Z$ of $Y$ so that
\begin{center}$H_1+Z_1=H_1\oplus Z_1$,\end{center}
and a projection $P_1: H_1\oplus Z_1\rightarrow H_1$ whose kernel is $Z_1$ and which satisfies
\begin{center}$\|P_1\|\leq 1+\gamma_1$.\end{center}
Assume now that, for some integer $j\in\mathbb{N}$, there exists $H_1,\cdots,H_j$ so that

\hspace{2.5cm}$d_{BM}(H_i,F_i)\leq C$ for any $i\in\{1,\cdots,j\}$,

\hspace{2.5cm}$H_1+\cdots+H_j=H_1\oplus\cdots\oplus H_j$,

\hspace{1.5cm} and $(H_1+\cdots+H_j)+Z_j=(H_1+\cdots+H_j)\oplus Z_j$,\\
with $P_j: (H_1+\cdots+H_j)\oplus Z_j \rightarrow H_1+\cdots+H_j$ a projection whose kernel is $Z_j$ and which satisfies
\begin{center}$\|P_j\|\leq\prod\limits_{i=1}^j(1+\gamma_i)$.\end{center}
Then, using the standard Mazur technique, we can find a subspace $Z_{j+1}$ of $Z_j$ which is a finite-codimensional subspace of $Y$, and
\begin{center}$\forall~y\in H_1+\cdots+H_j~~\forall~z\in Z_{j+1},~\|y\|\leq (1+\gamma_{j+1})\|y+z\|$.\end{center}
Since $X$ is $C$-crudely finitely representable in $Z_j$, there exists a finite-codimensional subspace $H_{j+1}$ of $Z_{j+1}$ so that $d_{BM}(H_{j+1},F_{j+1})\leq C$.\\
It is then immediate to conclude that the sequence $(H_j)_{j=1}^{\infty}$ satisfies the desired property.
\end{proof}

${}$

We can now prove our main result, Theorem \ref{P872}.

\begin{proof}
Let $A \in(0,\infty)$.
We adapt the proof of the Theorem $2.7$ of \cite{BL}.\\
Let $B_k=\{x \in M,~\|x\| \leq 2^{k+1}\}$ for $k \in \mathbb{Z}$.\\
Since $M$ is a proper subset of $X$, $B_k$ is a compact subset of $X$.\\
Let $G_{k,n}$ be a $\varepsilon_{k,n}$-net of $B_k$ containing $0$ (we will specify $\varepsilon_{k,n}$ later) , where $n \in \mathbb{N}$.\\
Let $\varphi_{k,n}$ be a map from $B_k$ to $G_{k,n}$ such that
$$\forall~x \in B_k,~~\|x-\varphi_{k,n}(x)\|=d(x,G_{k,n}).$$
Thereby:
$$\forall~x,y \in B_k,~ \|x-y\|-2\varepsilon_{k,n} \leq \|\varphi_{k,n}(x)-\varphi_{k,n}(y )\| \leq \|x-y\|+2\varepsilon_{k,n}.$$
By the previous lemma, we can construct a family of finite dimensional subspaces $(H_j)_{j=1}^{\infty}$ of $Y$ and linear isomorphisms $R_{k,j}: \textnormal{Span}(G_{k,j}) \rightarrow H_j$ such that $\|R_{k,j}\| \leq 1$ and $\|R_{k,j}^{-1}\| \leq C$, with $(H_j)_{j=1}^{\infty}$ FDD of $Z=\overline{\textnormal{Span}}(\bigcup_{j=1}^{\infty}H_j)$.\\
Let $\gamma > 0$. We will specify later the choice of $\gamma$.
We denote $P_n$ the projection of $Z$ on $H_1\oplus\cdots\oplus H_n$, and
using the previous lemma, we can assume that $\|P_n\| \leq 1+\gamma$ for any $n \in \mathbb{N}$.
We put, for $n \in \mathbb{N}$, $f_{k,n}=R_{k,n}\circ \varphi_{k,n}$. So:
\begin{center}$\forall~x,y \in B_k,~\frac{1}{C}(\|x-y\|-2\varepsilon_{k,n}) \leq \|f_{k,n}(x) -f_{k,n}(y)\| \leq \|x-y\|+2\varepsilon_{k,n}.$\end{center}
We define:

\hspace{2cm} $f_k : B_k \rightarrow \displaystyle\sum\limits_{n \geq 1}^{}H_n$

\hspace{2.9cm} $x \mapsto \displaystyle\sum\limits_{n=1}^{\infty}2^{-n}f_{k,n}(x)$\\
Let $\mu: (0,+\infty) \rightarrow (-\infty,0)$ be an increasing surjective continuous function, and $\sigma: (-\infty,0) \rightarrow (0,+\infty)$ the function defined by $\sigma(y) = \textnormal{inf}\{x \in (0,+\infty): \mu(x) \geq y\}$, with the convention inf~$ \emptyset = \infty$.\\
There exists $B\in [1,\infty)$ such that $\dfrac{1}{B}\leq A$.\\
We put $\varepsilon_{k,n}=\dfrac{1}{2B}\textnormal{min}(2^{-k},2^k)\times\textnormal{min}\left(\frac{\sigma(-n)}{\eta},1\right)$ where $\eta>2$ will be determined later. So:
\begin{align*}\left\|\displaystyle\sum\limits_{n=1}^{\infty}2^{-n}(f_{k,n}(x)-f_{k,n}(y))\right\| &\leq \left(\displaystyle\sum\limits_{n=1}^{\infty}2^{-n}\right)\|x-y\| + \dfrac{1}{B}\textnormal{min}(2^{-k},2^k)\displaystyle\sum\limits_{n=1}^{\infty}\left[2^{-n}\textnormal{min}\left(\frac{\sigma(-n)}{\eta},1\right)\right]\\
&\leq \|x-y\|+\dfrac{1}{B}\textnormal{min}(2^{-k},2^k).
\end{align*}

Thus: $\|f_k(x)-f_k(y)\| \leq \|x-y\|+\dfrac{1}{B}\textnormal{min}(2^{-k},2^k)$.

We define:

\hspace{2cm} $f: M \rightarrow Y$

\hspace{2.7cm} $x \mapsto \lambda_xf_k(x)+(1-\lambda_x)f_{k+1}(x)$,~ if $2^k \leq \|x\| \leq 2^{k+1}$ where $k \in \mathbb{Z}$,\\
with $\lambda_x = \frac{2^{k+1}-\|x\|}{2^k}$.\\
The rest of the proof will be divided into two parts.\\
In the first part, we estimate $\|f(x)-f(y)\|$ from above. The proof requires the inequality obtained for $f_k$, the definition of $f$ in terms of the dyadic slicing and the triangle inequality.\\
In the second part, an estimate of the compression function of $f$ is given.

${}$

\noindent \textbf{Part 1. Upper bound}\\
Let $x,y \in M$. Assume that $\|x\| \leq \|y\|$. Note that $f(0)=0$.

${}$

\underline{Case 1.} If $\|x\| \leq \frac{1}{2}\|y\|$ and $2^k\leq \|y\|<2^{k+1}$, where $k \in \mathbb{Z}$, then
\begin{align*}
\|f(x)-f(y)\| &\leq \|f(x)-f(0)\|+\|f(y)-f(0)\| \\
              &\leq \|x\| + \|y\|+2.2^{k+1} \\
              &\leq \frac{3}{2}\|y\|+4\|y\|\\
              &\leq 11(\|y\|-\|x\|)\\
              &\leq 11\|x-y\|.
\end{align*}

${}$

\underline{Case 2.} $\frac{1}{2}\|y\|<\|x\|\leq\|y\|$.\\
We consider two sub-cases.

${}$

\underline{Case 2.a.} $2^k\leq\|x\|\leq\|y\|<2^{k+1}$, where $k \in \mathbb{Z}$.\\
For $\lambda_x=\frac{2^{k+1}-\|x\|}{2^k}$ and $\lambda_y=\frac{2^{k+1}-\|y\|}{2^k}$, we have $|\lambda_x-\lambda_y|=\frac{\|y\|-\|x\|}{2^k}\leq\frac{\|x-y\|}{2^k}.$
Therefore
\begin{align*}
\|f(x)-f(y)\|&=\|\lambda_xf_k(x)-\lambda_yf_k(y)+(1-\lambda_x)f_{k+1}(x)-(1-\lambda_y)f_{k+1}(y)\|\\
             &\leq\lambda_x\|f_k(x)-f_k(y)\|+(1-\lambda_x)\|f_{k+1}(x)-f_{k+1}(y)\|+2|\lambda_x-\lambda_y|(\|y\|+2^{k+1})\\
             &\leq\lambda_x\left(\|x-y\|+\dfrac{1}{B}\right)+(1-\lambda_x)\left(\|x-y\|+\dfrac{1}{B}\right)+(2^{k+2}+2^{k+2})\dfrac{\|x-y\|}{2^k}\\
             &\leq 9\|x-y\|+\dfrac{1}{B}\\
             &\leq 9\|x-y\|+A.
\end{align*}

${}$

\underline{Case 2.b.} $2^k \leq \|x\| < 2^{k+1} \leq \|y\| < 2^{k+2}$, where $k \in \mathbb{Z}$.\\
Let $\lambda_x = \dfrac{2^{k+1}-\|x\|}{2^k}$ and $\lambda_y=\dfrac{2^{k+2}-\|y\|}{2^{k+1}}$.\\
We have $\lambda_x \leq \dfrac{\|x-y\|}{2^k}$, and so $\lambda_x\|x\| \leq 2\|x-y\|$.\\
In the same way: $1-\lambda_y = \frac{\|y\|-2^{k+1}}{2^{k+1}}\leq\frac{\|x-y\|}{2^{k+1}}$, and $(1-\lambda_y)\|y\|\leq 2\|x-y\|$.\\
It follows that:
\begin{align*}
\|f(x)-f(y)\|&=\|\lambda_xf_k(x)+(1-\lambda_x)f_{k+1}(x)-\lambda_yf_{k+1}(y)-(1-\lambda_y)f_{k+2}(y)\|\\
             &\leq \lambda_x(\|f_k(x)\|+\|f_{k+1}(x)\|)+(1-\lambda_y)(\|f_{k+1}(y)\|+\|f_{k+2}(y)\|)\\
             &{}\hspace{7cm}+\|f_{k+1}(x)-f_{k+1}(y)\|\\
             &\leq\lambda_x(2\|x\|+2^{k+2})+(1-\lambda_y)(2\|y\|+2^{k+3})+\|x-y\|+\dfrac{1}{B}\textnormal{min}(2^{-(k+1)},2^{k+1})\\
             &\leq 6\lambda_x\|x\|+6(1-\lambda_y)\|y\|+\|x-y\|+\dfrac{1}{B}\\
             &\leq 12\|x-y\|+12\|x-y\|+\|x-y\|+\dfrac{1}{B}\\
             &\leq 25\|x-y\|+A.
\end{align*}

${}$

\noindent \textbf{Part 2. Lower bound}\\
Let $x,y \in M$.\\
Assume that $\|x\|\leq\|y\|$.\\
We put, for $k \in \mathbb{Z}$ and $n \in \mathbb{N}$, $Q_{k,n}=R_{k,n}^{-1}\circ \Pi_n$, where $\Pi_1 = P_1$, and $\Pi_n=P_n-P_{n-1}$ when $n$ is greater than or equal to $2$.\\
For the sequel, we also assume that $2^k\leq \|x\| < 2^{k+1}$ and $2^l\leq \|y\| < 2^{l+1}$, where $(k,l)\in \mathbb{Z}^2$.\\
\textbullet~ Assume first that $\|x-y\|<\sigma(-1)$.\\
Let then $n \in \mathbb{N}$ such $$\sigma(-(n+1))\leq \|x-y\|<\sigma(-n),$$
in other words $$-(n+1)\leq \mu(\|x-y\|)<-n.$$
Then:
\begin{align*}
Q_{k,n+1}(f(x))&= 2^{-(n+1)}\lambda_x\varphi_{k,n+1}(x),\\
Q_{k+1,n+1}(f(x))&= 2^{-(n+1)}(1-\lambda_x)\varphi_{k+1,n+1}(x),\\
Q_{l,n+1}(f(y))&= 2^{-(n+1)}\lambda_y\varphi_{l,n+1}(y),\\
Q_{l+1,n+1}(f(y))&= 2^{-(n+1)}(1-\lambda_y)\varphi_{l+1,n+1}(y),\\
\textnormal{and } Q_{r,n+1}(f(x))&= Q_{s,n+1}(f(y))=0 \textnormal{ for } r \notin \{k,k+1\},~ s \notin \{l,l+1\}.
\end{align*}
Therefore:
\begin{align*}
(Q_{r_1,n+1}+\cdots+Q_{r_s,n+1})(f(x)-f(y))=& 2^{-(n+1)}[\lambda_x\varphi_{k,n+1}(x)+(1-\lambda_x)\varphi_{k+1,n+1}(x)\\
                                            & -\lambda_y\varphi_{l,n+1}(y)-(1-\lambda_y)\varphi_{l+1,n+1}(y)]
\end{align*}
with $s \in \{2,3,4\}$ and $r_1,\cdots,r_s \in \{k,k+1,l,l+1\}$.\\
We will now use the following inequality:
$$\underset{r \in \{k,k+1,l,l+1\}}{\textnormal{max}}\|\varphi_{r,n+1}(x)-x\|\leq\frac{\sigma(-(n+1))}{\eta}.$$
This allows us to get:
\begin{align*}
\|(Q_{r_1,n+1}+\cdots+Q_{r_s,n+1})(f(x)-f(y))\| &\geq 2^{-(n+1)}\left(\|x-y\|-\frac{2\sigma(-(n+1))}{\eta}\right) \\
&\geq 2^{-(n+1)}\frac{\eta-2}{\eta}\|x-y\|.
\end{align*}
As $\|Q_{r_1,n+1}+\cdots+Q_{r_s,n+1}\|\leq 8C(1+\gamma)$, we deduce
$$\|f(x)-f(y)\|\geq \frac{2^{-(n+1)}(\eta-2)\|x-y\|}{8\eta C(1+\gamma)} \geq\frac{2^{\mu(\|x-y\|)}(\eta-2)\|x-y\|}{16\eta C(1+\gamma)}.$$
\textbullet~ Assume now that $\|x-y\| \geq \sigma(-1)$ or equivalently that $-1 \leq \mu(\|x-y\|)$.\\
So, using the maps $Q_{k,1},Q_{k+1,1},Q_{l,1}$ and $Q_{l+1,1}$ instead of $Q_{k,n+1},Q_{k+1,n+1},Q_{l,n+1}$, and $Q_{l+1,n+1}$, we obtain that
$$\|f(x)-f(y)\|\geq\frac{(\eta-2)\|x-y\|}{16\eta C(1+\gamma)}.$$
Since $\mu \leq 0$, it follows that
$$\|f(x)-f(y)\|\geq\frac{2^{\mu(\|x-y\|)}(\eta-2)\|x-y\|}{16\eta C(1+\gamma)}.$$
Since $\gamma$ can be chosen as small and $\eta$ as large as desired, thanks to Lemma \ref{L123}, this concludes our proof.
\end{proof}

\begin{Corollary} Let $X$ and $Y$ be two Banach spaces.\\
If $X$ is c.f.r in the finite-codimensional subspaces of $Y$, then for any $Z \in\textnormal{cof}(Y)$, for any proper subset $M$ of $X$, $M$ almost bi-Lipschitz embeds into $Z$.
\end{Corollary}

\begin{Remark}
Corollary 2.8 from \cite{BL} is a special case of this last result. This corollary says that, for $M$ a proper subset of $L_p$ ($p \in [1,\infty]$), and $X$ a Banach space containing uniformly the $\ell_p^n$'s, $M$ almost bi-Lipschitz embeds into $X$. Indeed, the following assertions are equivalent, for an infinite-dimensional Banach space $X$ and $p \in [1,\infty]$:\\
$(i)$~ $X$ contains the $\ell_p^n$'s uniformly.\\
$(ii)$~ $L_p$ is finitely representable in $X$.\\
$(iii)$~ $L_p$ is c.f.r in the finite-codimensional subspaces of $X$.\\
A proof is given in \cite{Ne}.
\end{Remark}

\begin{Definition}
A Banach space $X$ is said to be \textit{locally minimal} if there exists a constant $K>1$ such that $X$ is $K$-c.f.r in each of its infinite-dimensional subspaces.
\end{Definition}

\begin{Examples}
$c_0$, $\ell_p$ for $1\leq p<\infty$, and the Tsirelson space $T^*$ are such spaces.
\end{Examples}

\noindent With this last definition, the following corollary is an obvious consequence of Theorem \ref{P872}:

\begin{Corollary}
Let $X$ and $Y$ be two Banach spaces such that $Y$ is locally minimal, and $X$ is finitely representable in $Y$.\\
Let $M$ be a proper subset of $X$.
Then $M$ almost bi-Lipschitz embeds into $Y$.
\end{Corollary}





\end{document}